\documentclass[11pt]{amsart}

\usepackage{amscd}
\usepackage{amsmath,latexsym,amssymb,amsthm}
\usepackage{hyperref}
\usepackage{lscape,fancyhdr}
\usepackage[top=2.5cm, left=3cm, right=3cm]{geometry}
\usepackage{epsfig,graphicx}
\usepackage{float}
\allowbreak
\newtheorem{secn}{Definition}[section]
\newtheorem{thm}[secn]{Theorem}

\newtheorem{prop}[secn]{Proposition}
\newtheorem{lem}[secn]{Lemma}
\newtheorem{defn}[secn]{Definition}

\newtheorem{rem}[secn]{Remark}

\newcommand\EE{{\mathcal E}}
\newcommand\FF{{\mathcal F}}

\newcommand\HH{{\mathcal H}}

\newcommand\LL{{\mathcal L}}
\newcommand\MM{{\mathcal M}}

\newcommand\OO{{\mathcal O}}
\newcommand\PP{{\mathcal P}}

\newcommand\SSS{{\mathcal S}}
\newcommand\TT{{\mathcal T}}
\newcommand\UU{{\mathcal U}}

\newcommand\YY{{\mathcal Y}}

\newcommand\PMF{{\PP\kern-2pt\MM\FF}}

\newcommand\PML{{\PP\kern-2pt\MM\LL}}

\newcommand\ep{\epsilon}
\newcommand\hhat{\widehat}

\newcommand{\fsubd}{\mathrel{{\scriptstyle\searrow}\kern-1ex^d\kern0.5ex}}
\newcommand{\bsubd}{\mathrel{{\scriptstyle\swarrow}\kern-1.6ex^d\kern0.8ex}}
\newcommand{\fsubeq}{\mathrel{\raise-.7ex\hbox{$\overset{\searrow}{=}$}}}
\newcommand{\bsubeq}{\mathrel{\raise-.7ex\hbox{$\overset{\swarrow}{=}$}}}

\newcommand{\bbar}{\overline}

\newcommand{\tsh}[1]{\left\{\kern-.9ex\left\{#1\right\}\kern-.9ex\right\}}

\begin{document}
\title{Complex of Relatively Hyperbolic Groups}
\author{Abhijit Pal}
\thanks{Research of the first author is supported by INSPIRE Research Grant}
\address{Department of Mathematics and Statistics, India Institute of Technology-Kanpur}
\email{abhipal@iitk.ac.in}
\author{Suman Paul}
\address{Department of Mathematics and Statistics, India Institute of Technology-Kanpur}
\email{sumanpl@iitk.ac.in}
\date{}
\maketitle
\begin{abstract}
In this article, we prove a combination theorem for a complex of relatively hyperbolic groups. It is a generalization of Martin's \cite{martin}
work for combination of hyperbolic groups over a finite $M_K$-simplicial complex, where $k\leq 0$.
\end{abstract}
\begin{center}{AMS Subject Classification : 20F65, 20F67, 20E08}\end{center}
\section{Introduction}

In \cite{dah}, Dahmani showed that if $G$ is the fundamental group of an acylindrical finite graph of relatively hyperbolic groups
with edge groups fully quasi-convex in the respective vertex groups, then $G$ is hyperbolic relative to the images of the maximal parabolic
subgroups of vertex groups and their conjugates in $G$. By gluing the relative hyperbolic boundaries of each local groups, Dahmani constructed
a compact metrizable space $\partial G$ on which $G$ has convergence action and the limit points are either conical or bounded parabolic.
So, $G$ is a relatively hyperbolic group due to Yaman \cite{yaman}. Using these ideas, Martin \cite{martin} generalized this combination
theorem for complex of hyperbolic groups. Let $G(\YY)$ be a strictly developable non-positively
curved simple complex of groups over a finite $M_k$ simplicial complex with $k\leq 0$. Let $G$ be the fundamental group of $G(\YY)$
and $X$ be a universal covering of $G(\YY)$. Martin, in \cite{martin}, proved that if $X$ is hyperbolic, local groups are hyperbolic,
local maps are quasiconvex embeddings and the action of $G$ on $X$ is acylindrical 
(i.e. there exists $K>0$ such that any pair of points of diameter at least $K$ in $X$ has finite pointwise stabilizer, see Definition \ref{acy act}), then $G$ is hyperbolic. 
In this article, we prove a relative hyperbolic version of Martin's result.

\begin{thm}\label{main theorem}
 Let $\mathcal{G(Y)}$ be a strictly developable simple complex of finitely generated groups over a finite $M_{\kappa}$-simplicial complex $Y$ 
 with $k\leq 0$
 and satisfying the following properties:
\begin{itemize}
  \item For each vertex $v$ of $Y$, the vertex group $G_{v}$ is  relatively hyperbolic to a maximal parabolic subgroup ${P}_{v}$.
  \item Local maps $\phi_{\sigma,\sigma^{'}}$ are fully quasi-convex embeddings i.e. if $\sigma\subset\sigma'$ 
  then $\phi_{\sigma,\sigma^{'}}(G_{\sigma'})$
  is fully quasiconvex in $G_{\sigma}$,
  \item The universal covering $X$ of $\mathcal{G(Y)}$ is hyperbolic.
  \item The action of $G$, the fundamental group of $\mathcal{G(Y)}$, on $X$ is acylindrical.
\end{itemize}   
Then $G$ is hyperbolic relative to $\mathcal{P}$, where $\mathcal{P}$ is the collection of the images of ${P}_{v}$ in $G$
under the natural embedding $G_v\hookrightarrow G$ and there conjugates in $G$. Furthermore, local groups are fully quasiconvex in $G$.
\end{thm}

The combination theorem of this sort for finite graph of hyperbolic groups was first given by Bestvina \& Feighn (\cite{best}). Here, the 
edge groups embed quasi-isometrically into vertex groups and the graph of groups satisfies `hallway flare condition'.
This combination theorem was generalized by Mj. \& Reeves (\cite{Mj com}) for relatively hyperbolic case.
Further, Mj.\& Sardar in \cite{Mj Sardar} generalized these combination theorems for metric bundles with base space hyperbolic
and fibers are uniformly hyperbolic metric spaces. Let $S$ be a closed (not closed) orientable surface of negative Euler characteristic and
$\phi:S\to S$ be a pseudo-Anosov homomorphism (fixing punctures and boundary pointwise, if non-empty).
Let $M_{\phi}$ be the mapping torus, then it follows from combination theorem of Bestvina \& Feighn in \cite{best} (Mj.\& Reeves in \cite{Mj com})
that the fundamental group $\pi_1(M_{\phi})$ is hyperbolic (relatively hyperbolic). However, in this case $\pi_1(S)$ being infinite index normal subgroup
in $\pi_1(M)$ is not quasiconvex and $\pi_1(M_\phi)$ does not acts acylindrically on the Bass-Serre tree $\mathbb R$ (real numbers).
Ilya Kapovich \cite{kapo} proved that if finite graph of hyperbolic groups is acylindrically hyperbolic and satisfies
quasi-isometrically embedded condition then fundamental group of graph of groups is hyperbolic and vertex groups are quasiconvex
in the fundamental group. This was generalized to finite graph of relatively hyperbolic groups by Dahmani \cite{dah}.
Martin in \cite{martin} generalized Ilya Kapovich's theorem for finite complex of hyperbolic groups. 

An example to our interest can be constructed from Osin and Minasyan's work \cite{osin}.
 For instance,  let $M$ be a $3$-manifold
and $S$ be a punctured torus embedded in $M$. Suppose $M$ splits over $S$ with $M\setminus S$ having two components.
Suppose $G$ is fundamental group of $M$, $A,B$ are fundamental groups of components and $C$ is the fundamental group of punctured torus.
Then $G=A*_CB$. Let $C_1$ be the (cyclic) peripheral  subgroup of $C$ with respect to which $C$ is hyperbolic relative to $C_1$.
Now if $A$ is hyperbolic relative to the subgroup $C_1$ and $C$ is relatively quasiconvex in $A$.
Then $G$ is acylindrically hyperbolic. Further, if $B$ is hyperbolic relative to $C_1$
and $C$ is fully quasiconvex in both $A,B$ then $G$ is hyperbolic relative to the collection of conjugates of $C_1$ in $G$.

We will adapt the strategies followed by Dahmani and Martin to prove the main theorem which is as follows
\begin{itemize}
\item
 In our case local groups are relatively hyperbolic. In order to get
a hyperbolic space on which the relatively hyperbolic group acts properly discontinuously, we will attach `combinatorial horoballs'
to each cosets of the peripheral subgroup. The resulting space is called Augmented space (See \ref{ch},\ref{aug space}, \ref{GM}).\\
We will construct a complex of spaces, $EG$ (reps. boundary, $\partial G$) gluing the augmented spaces(see Definition \ref{aug space}) 
(resp. Bowditch boundaries) of the local groups similar to Martin's paper\cite{martin}.  
In Martin's paper local spaces are hyperbolic spaces on which local groups 
acts properly discontinuously and cocompactly. Here, we will take local spaces as augmented spaces and
use the fact that relatively hyperbolic group acts on Bowditch boundary by convergence action. 
The topology defined on $EG \cup \partial G$ in \cite{martin}
will work in our case and it will make $EG \cup \partial G$, a compact metrizable space.

\item
Next we will prove that the action of $G$ on $\partial G$ is by convergence action. 
Since there are parabolic limit points in boundaries of local groups, we have to modify the  proofs in \cite{martin} to work in our case.
\item
Lastly we will show that all the limit points for this convergence action is either conical or bounded parabolic. Then by the Theorem \ref{yaman thm}(due to Yaman, \cite{yaman}), $G$ will be relatively hyperbolic. 
  
\end{itemize}

In section \ref{RHCA}, we will give several definitions of relatively hyperbolic groups due to Farb, Grooves \& Manning and Bowditch.
Convergence action, fully quasiconvex subgroups , convergence property and finite intersection properties are given in this section.
 Complex of groups is described in section \ref{CS} and in the subsequent section \ref{EG}, the construction
of boundary $\partial G$ of fundamental group $G$ of complex of groups $\mathcal{G(Y)}$ is provided. In section \ref{convergence} and \ref{MT},
we will prove Theorem \ref{main theorem}.\\ \\
\textbf{Acknowledgement:} We are grateful to the anonymous reviewer for carefully reading the manuscript and helping to improve the exposition of this article.

\section{Preliminaries on Relative Hyperbolicity}\label{RHCA}

\subsection{Relative Hyperbolicity}
Relatively hyperbolic groups were first introduced by Gromov 
\cite{gro} to study hyperbolic manifolds with cusps. It was then studied by several people, we refer to the article \cite{hrus} by 
Hruska for several equivalent notions of relatively hyperbolic groups. For our purpose, we will require three equivalent definitions of 
relative hyperbolicity due to Farb \cite{farb}, Bowditch \cite{bow} and Groves \& Manning \cite{groves}. 
\begin{defn}(Hyperbolic Metric Space)\label{geod hyp}
  Let $\delta\geq 0$. We say that a geodesic triangle $\Delta$ is $\delta$-\textit{slim } in a geodesic metric space
 if any side of the triangle $\Delta$ is contained in the $\delta$- neighbourhood of the union of the other two sides.
A geodesic metric space is said to be $\delta$-\textit{hyperbolic} if all the triangles are $\delta$-slim. 
A geodesic metric space is said to be hyperbolic if it is $\delta$-hyperbolic for some $\delta\geq 0$.
\end{defn}

First we give the definition of relative hyperbolicity due to Farb. Let $G$ be a finitely generated group and $H$ be a finitely generated subgroup of it.
Also let $\Gamma_G$ be the Cayley graph of $G$.

\begin{defn}{(Coned-off Cayley Graph, \cite{farb})}
 The coned-off Cayley graph of $G$ w.r.t. $H$, denoted by $\widehat{\Gamma_G}$, 
is obtained from $\Gamma_G$ by adding an extra vertex $v(gH)$ for each left coset of $H$ in $G$ and an
extra edge $e(gh)$ of length 1/2 joining each $gh\in gH$ to $v(gH)$.
\end{defn}
Given a path $\gamma$ in $\Gamma_G$, the inclusion $\Gamma_G\rightarrow \widehat{\Gamma_G}$, gives a path $\tilde{\gamma}$ 
(after removing backtracks and loops of length 1) in $\widehat{\Gamma_G}$. If $\tilde{\gamma}$ goes through some $v(gH)$, then we 
say $\gamma$ penetrates $gH$. We call $\gamma$ to be a relative $k$-quasi geodesic if $\tilde{\gamma}$ is a $k$-quasi geodesic in $\widehat{\Gamma_G}$.
Also $\gamma$ is said to be a path without backtracking if after going through a cone point $v(gH)$ it never return to $gH$.

\begin{defn}{(Bounded Coset Penetration Property,\cite{farb})}
$(G,H)$ is said to have bounded coset penetration property if for each $k>1$ there exists $c(k)>0$ such that for any two relative $k$-quasi geodesics $\gamma_{1}, \gamma_{2}$ in $\Gamma_{G}$ with $d_{\Gamma_{G}}(\gamma_{1}, \gamma_{2})\leq 1$, the following holds,

(1) if $\gamma_{1}$ penetrates $gH$ but $\gamma_{2}$ does not then $\gamma_{1}$ travels at most $c(k)$ distance in $gH$.
\end{defn}

(2) if both $\gamma_{1}, \gamma_{2}$ penetrates $gH$ then the entry points as well as the exit points of the paths are $c(k)$ close to each other in $\Gamma_{G}$.
\begin{defn}{(B.Farb,\cite{farb})}
Let $G$ be a finitely generated group and $H$ be a finitely generated subgroup of it. $G$ is said to strongly hyperbolic relative to $H$ if $\widehat{\Gamma_G}$ is hyperbolic and $(G,H)$ satisfy bounded coset penetration property.
\end{defn}
The next definition by Bowditch gives a dynamical characterization of relative hyperbolicity which we will essentially use to prove the main theorem.
For that we need the notion of convergence group.

\begin{defn}{(Convergence Group)}
Let $G$ acts on compact metrizable space $M$. The action is called  convergence group action if 
for any sequence $\{g_{n}\}$ in $G$, there exists a subsequence $\{g_{\phi(n)}\}$ and $\xi^{+},\xi^{-}\in M$ such that $g_{\phi(n)}(K)$ 
converges uniformly to $\xi^{+}$, for all compact sets  $K\subset M\backslash \{\xi^{-}\}$.
\end{defn}

\begin{defn}\begin{enumerate}
 \item {(Bounded Parabolic Limit Points)}
An element $g\in G$ is called parabolic if it fixes exactly one point of $M$ and the corresponding fixed point $\xi$(say) is said to be
parabolic limit point. Furthermore, a parabolic limit point is said to be bounded parabolic if $Stab(\xi)$ acts properly discontinuously 
and cocompactly on $M\backslash \{\xi\}$.
\item {(Conical Limit Point)} Let $G$ has a convergence action on $M$. A point $\xi\in M$ is said to be conical limit point if there exists a sequence 
$\{g_{n}\}$ and $\xi^{+}\neq\xi^{-}\in M$ such that 
$g_{n}\xi \rightarrow \xi^{+}$, $g_{n}\xi^{'} \rightarrow \xi^{-}$ for all $ \xi^{'}\in M\backslash \{\xi\}$.
\item {(Geometrically Finite Action)} Let $G$ has a convergence action on a compact metrizable space $M$. The action is said to be 
geometrically finite if the limit points are either conical or bounded parabolic.
             
     \end{enumerate}

\end{defn}
Next we give the Bowditch's definition of Relative Hyperbolicity.

\begin{defn}{(Bowditch,\cite{bow})}\label{bowditch}
Let $G$ be finitely generated group and $\PP$ be a finite collection of finitely generated subgroups of it. 
$G$ is said to hyperbolic relative to $\PP$ if it acts properly discontinuously on a proper hyperbolic metric space $\widetilde{\Gamma}$ such that
\begin{itemize}
  \item $G$ acts on $\partial \widetilde{\Gamma}$ by convergence and geometrically finite action.
  \item the conjugates of the elements of $\PP$ are precisely the maximal parabolic subgroups.
\end{itemize}
we call $\partial \widetilde{\Gamma}$ the Bowditch boundary of $G$.
\end{defn}

Note that $\widehat{\Gamma_G}$ is locally infinite and the action of $G$ on it, is not properly discontinuous unless $H$ is finite. 
Groves \& Manning have defined a proper metric space by gluing combinatorial horoballs along parabolic subgroups and their translates, 
similar to coned-off Cayley graph and it is called Augmented space. Also $G$ acts on its augmented space properly discontinuously by isometries.
Let $G$ be finitely generated group and $\PP$ be a finite collection of subgroups of it. 
Let $\SSS{}$ be a finite  generating set of $G$ such that $\big<\SSS{}\cap P\big> = P$  for all $P\in\PP$ and 
$\Gamma_G$ be the Cayley graph of $G$ with respect to  $\SSS{}$.

\begin{defn}\label{ch}{(Combinatorial Horoballs, \cite{groves})}
Let $C$ be a 1-complex with $0$-skeleton $C^{0}$ and $1$-skeleton $C^1$.  We will construct a 1-complex $\HH{}(C)$ following ways:
\begin{itemize}
 \item $0$-skeleton of $\HH(C)$, $\HH{}(C)^{(0)}:= C^{(0)}\times (\{0,1,2,...\}) $,
 \item $1$-skeleton of $\HH(C)$,
$\HH{}(C)^{(1)}:=\big\{[(v,0),(w,0)] : v,w\in C^{(0)},
[v,w]\in C^{(1)}\big\}\cup \big\{[(v,k),(w,k)] : v,w\in C^{(0)}, k>0, d_{C}(v,w)\leq 2^{k} \big\}\cup \big\{[(v,k),(v,k+1)] : v\in C^{(0)},
k\geq 0\big\}$.
\end{itemize}
\end{defn}

\begin{defn}{(Augmented Space, \cite{hrus})}\label{aug space}
Let $G,\PP, \SSS{}$ be as mentioned above. Also let $\TT$ be the set of representative for distinct cosets 
of all $P\in \PP$. The Cayley graph of $P$ with respect to $P\cap \SSS{}$ embedded in $\Gamma_G$ as a subcomplex.
Let $\Gamma_{t}$, $t\in \TT$, be the translates of these subcomplexes. We define

$$\Gamma^{h}_G:=\Gamma_G\cup \big( \cup_{t\in\TT}(\HH(\Gamma_{t}))\big)\big /\simeq$$
as augmented space, where $\HH(\Gamma_{t})\times\{0\}$'s are identified to subcomplexes $\Gamma_{t}$.

\end{defn}
\begin{defn}\label{GM}{(Groves \& Manning,\cite{groves})}
$G$ is said to hyperbolic relative to $\PP$ if the augmented space $\Gamma^{h}_G$ is hyperbolic for any appropriate choice of $\SSS{}$.
\end{defn}

\begin{rem}
Due to equivalence of these  definitions  we can take $\Gamma^{h}$ as $\tilde{\Gamma}$ and
$\partial\Gamma^{h}$ will be Bowditch boundary.
\end{rem}

Next we will state a theorem due to A. Yaman which is a generalization of Bowditch's result on characterization of hyperbolic groups\cite{bow2}.

\begin{thm}\label{yaman thm}(A.Yaman, \cite{yaman})
Let $G$ has a geometrically finite action on a perfect metrizable compact space $M$ and $\PP{}$ be the collection of maximal parabolic subgroups. 
Also let every parabolic subgroup be finitely generated and there are only finitely many orbits of bounded parabolic points. 
Then $G$ is hyperbolic relative to $\PP{}$ and $M$ is equivariantly homeomorphic to its Bowditch boundary. 
\end{thm}
 
We can omit the finiteness of the set of orbits of parabolic points by a theorem of Tukia(\cite{tukia}, Theorem 1B). 
As discussed in the introduction we will use this characterization of relative hyperbolicity to prove the main theorem.

\subsection{Fully quasi-convex subgroup}

Fully quasi-convex subgroups of relatively hyperbolic group were introduced by Dahmani in \cite{dah}. 
It is a generalization of quasi-convex subgroups of hyperbolic group in the sense that it satisfies \textit{limit set property, 
convergence property and finite intersection (finite height) property} which is not in general true for quasi-convex subgroup of
relatively hyperbolic group. The definition of fully quasi-convex subgroups, Remark \ref{quasiconvex-parabolic} and Theorems \ref{limset}
and \ref{conv} are taken from \cite{dah}. We refer to section 1.2 of \cite{dah} for proofs.

\begin{defn} (Dahmani, \cite{dah})
Let $G$ be a relatively hyperbolic group with Bowditch boundary $\partial G$. A subgroup $H$ of $G$ is called quasi-convex if 
$H$ has a geometrically finite action on $\Lambda H$. It is called fully quasi-convex if for any infinite sequence $\{g_{n}\}$, 
all comes from distinct cosets of $H$,  $\bigcap_{n}(g_{n}\Lambda H)$ is empty. 

\end{defn}

\begin{rem}
If $H$ is fully quasiconvex, then $gHg^{-1}$ is also fully quasi-convex, for all $g\in G$.

\end{rem}

\begin{rem}{\cite{dah}}\label{quasiconvex-parabolic}
 Let $G$ be a relatively hyperbolic group. If $H$ is fully quasi-convex in $G$, 
 then each parabolic point for $H$ in $\Lambda(H)$ is a parabolic point for $G$ in $\partial G$ and if $P$ is 
 the corresponding maximal parabolic subgroup in $G$ then the corresponding maximal parabolic subgroup in $H$ is precisely $P\cap H$. 
\end{rem}
 The following two properties of fully quasi-convex subgroups are proved by F. Dahmani \cite{dah}. 
 
\begin{thm}(Limit set property,\cite{dah})\label{limset}
Let $H_{1}$ and $H_{2}$ are fully quasi-convex in $G$ then $H_{1}\bigcap H_{2}$ is fully quasi-convex.
Moreover $\Lambda (H_{1}\bigcap H_{2})= \Lambda H_{1}\bigcap \Lambda H_{2}$.
\end{thm}

\begin{thm}(Convergence property,\cite{dah})\label{conv}
Let $G$ be a relatively hyperbolic group and $H$ be a fully quasi-convex subgroup in it. Let $\{g_{n}\}$ be 
a sequence of elements in $G$ all comes from distinct cosets of $H$. Then there exists a subsequence $\{g_{\phi (n)}\}$ 
such that $g_{\phi (n)}\Lambda H$ uniformly converges to a point.
\end{thm}

\begin{lem}\label{subgroup qc}
Let $H$ be a finitely generated fully quasi convex subgroup of 
finitely generated relatively hyperbolic group $G$. Then $\Gamma^{h}_{H}$ is quasi convex in $\Gamma^{h}_{G}$.
\end{lem}
\begin{proof}
Let $H$ be generated by $S$ and extend the generating set to generate $G$. 
Then the corresponding augmented spaces $\Gamma^{h}_{H}$ will be a subgraph of $\Gamma^{h}_{G}$. 
Take two points $x,y$ in $\Gamma^{h}_{H}$ and join them by a geodesic $c$ in $\Gamma^{h}_{G}$.
Since $H$ is relatively quasi convex in $G$, by a theorem of Hruska\cite{hrus}, $\Gamma_H$ is quasi convex in $\widehat{\Gamma}_G$.
Let $\hhat c$ be the image of $c$ in $\hhat{\Gamma}_G$ after removing the backtracks and loops of length $1$.  
Take the projection of $\hhat c$ onto $\Gamma_H$ and call it $\tilde{c}$, hence $\tilde{c}$ is a quasi geodesic in $\widehat{\Gamma}_G$.
Note that image of $\hhat c$ and $c$ are same outside horoballs. Also Hausdorff distance between $\hhat c$ and $\tilde{c}$ 
is bounded outside horoballs and if they enter same horoballs then the distances between entry points, as well as exit points, are bounded 
in $\Gamma_{G}$ and so in $\Gamma^{h}_{G}$. So if we can prove for any geodesic in $\Gamma^{h}_{G}$ entirely
lies in a horoball with starting and ending points close to $H$ the geodesic is in bounded distance from $\Gamma^{h}_{H}$, then we are done.
By Lemma 3.1 of Grooves and Manning\cite{groves} any geodesic in a horoball tracks a geodesic consists of two vertical 
segments and one horizontal segment(Hausdorff distance is at most 4). Now a geodesic consists of two vertical segments and 
one horizontal segment with starting and ending points close to $H$ lie in a bounded neighbourhood of $\Gamma^{h}_{H}$. Hence we are done. 
\end{proof}

Next we will prove that there are finitely many conjugates of a fully quasi-convex subgroup which have infinite total intersection.

\begin{prop}(Finite intersection property)\label{finite}
Let $G$ be a relatively hyperbolic group and $H$ be a fully quasi-convex subgroup in it. 
Then there exists finitely many distinct left cosets $g_{1}H,g_{2}H...g_{m}H$ in $G$ for which 
$\bigcap\limits_{k=1}^m g_{k}Hg_{k}^{-1}$ is infinite.
\end{prop}
\begin{proof}
If possible, let there exists a infinite sequence $\{g_{n}\}$ all comes from distinct cosets of $H$ such that 
$\bigcap\limits_{n=1}^{\infty} g_{n}Hg_{n}^{-1}$ is infinite, i.e. $\Lambda (\bigcap\limits_{n=1}^{\infty} g_{n}Hg_{n}^{-1})$ is non empty. 
But $\Lambda (\bigcap\limits_{n=1}^{\infty} g_{n}Hg_{n}^{-1})\subset\bigcap\limits_{n=1}^{\infty}\Lambda(g_{n}Hg_{n}^{-1}) $
and the fact that $\Lambda (gHg^{-1})=g \Lambda H$, we have $\bigcap\limits_{n=1}^{\infty}g_{n} \Lambda H$ is 
non empty which contradicts the second condition of fully quasi-convexity.

\end{proof}

\section{Background on Complex of Groups}\label{CS}

H. Bass and J.
P. Serre in \cite{trees} completely described the class of groups which act on trees without inversion.
Such groups are fundamental group of graph of groups. A. Haefliger in \cite{haefliger} generalized this theory to the class of groups 
acting on simplicial complexes and it is called complex of groups. In this section we will discuss the basics of complex of groups. 
For a detailed discussion on this topic we refer to \cite{brid}.

Let $Y$ be a simplicial complex. We will denote the set of simplices and set of vertices of $Y$ by $S(Y)$ and $V(Y)$ respectively.
Let $\YY$ be the scwol (refer to \cite{brid}) corresponding to the first Barrycentric subdivision of $Y$ and its directed edge 
set is denoted by $\EE^{\pm}(\YY{})$.
\subsection{Complex of Groups}
\begin{defn}{(Complex of Groups, \cite{brid})}
A simple complex of groups, $G(\mathcal{Y})$, over a simplicial complex $Y$ consists of 
\newline{(1) local groups $G_{\sigma}$ for each $\sigma \in S(Y)$}
\newline{(2) a monomorphism $\varphi_{\sigma,\sigma^{'}}:G_{\sigma^{'}}\rightarrow G_{\sigma}$ whenever $\sigma \subset \sigma^{'}$.}
\newline{(3) for $\sigma \subset \sigma^{'} \subset \sigma^{''}$,  
$\varphi_{\sigma,\sigma^{''}}=\varphi_{\sigma,\sigma^{'}} \circ \varphi_{\sigma^{'},\sigma^{''}}$ }
\end{defn}

\begin{defn}{(Fundamental Group of Complex of Groups,  \cite{brid})}\label{fundaG}
Let $T$ be a maximal tree in 1-skeleton of $\mathcal{Y}$. The Fundamental Group of $G(\mathcal{Y})$ with respect to $T$, 
denoted by $\pi_{1}(G(\mathcal{Y}),T)$, is generated by $\bigsqcup\limits_{\sigma\in S(Y)} G_{\sigma} \bigsqcup \mathcal{E}^\pm(\mathcal{Y})$ subject to  \\   
 (1) relations of $G_{\sigma}$,\\
 (2) $(a^{+})^{-1}=a^{-}, ({a}^{-})^{-1}=a^{+}$,\\
 (3) $(ab)^{+}=a^{+}b^{+}$\\
 (4) $a^{+}ga^{-}=\varphi_{a}(g)$, \\
 (5) $a^{+}=1 $ for all edge $a$ of $T$.
 
\end{defn}

In fact the above definition is independent of the choice of the maximal tree and we will call it $G$ in the subsequent sections.
There is a canonical morphism, $\iota_{T}: G(\YY{})\rightarrow G$ which takes $G_{\sigma}\rightarrow G$ and $a\mapsto a^{+}$.
The natural homomorphisms $G_{\sigma}\rightarrow G$ is injective if and only if the complex of groups $G(\YY)$ is developable.
For definition of developability, see Definition 2.11 of \cite{brid}.

Next we will define a CW complex on which $G$ will act naturally and the quotient space will be $Y$.
\begin{defn}{(Universal Covering, \cite{brid})}
We define the universal covering of $G(\YY{})$ associated to $\iota_{T}$ as
\begin{center}
$X:=\Bigg(G\times \coprod_{\sigma\in S(Y)}\sigma\Bigg)\bigg/ \simeq$
\end{center}
where $(g,i_{\sigma,\sigma^{'}}(x))\simeq (g\iota_{T}([\sigma,\sigma^{'}])^{-1},x)$, $[\sigma,\sigma^{'}]\in \EE{}(\YY{})$,
$i_{\sigma,\sigma^{'}}:\sigma^{'}\rightarrow \sigma$ is the embedding and $(gg^{'},x)\simeq (g,x)$, $g^{'}\in G_{\sigma},g\in G$.
\end{defn}
$G$ acts naturally on $X$ by left multiplication on the first factor.

\begin{defn}{(Acylindrical Action)}\label{acy act}
Let $K>0$. The action of $G$ on a metric space $(X,d)$ is said to be $K$-\textit{acylindrical} if for any pair of points  $x,y\in X$ with  $d(x,y)\geq K$ the pointwise stabilizer 
of $\{x,y\}$ is finite. The action of $G$ on $X$ is said to be acylindrical if it is $K$-acylindrical for some $K>0$.
\end{defn}

\subsection{Complex of Spaces}
\begin{defn}
A complex of spaces, $C(\YY{})$, over a simplicial complex $Y$ consists of 
\newline{(1) local spaces $C_{\sigma}$ for each $\sigma \in S(Y)$}
\newline{(2) an embedding  $\varphi_{\sigma,\sigma^{'}}:C_{\sigma^{'}}\rightarrow C_{\sigma}$ whenever $\sigma \subset \sigma^{'}$.}
\newline{(3) for $\sigma \subset \sigma^{'} \subset \sigma^{''}$,  $\varphi_{\sigma,\sigma^{''}}=\varphi_{\sigma,\sigma^{'}}
\circ \varphi_{\sigma^{'},\sigma^{''}}$ }
\end{defn}

\begin{defn}{(Realization of complex of spaces)}
Let $C(\YY{})$ be a complex of spaces over $Y$. We define the realization of $C(\YY{})$ to be the quotient space
\begin{center}
$|C(\YY{})|:=\Big( \coprod_{\sigma\in S(Y)} (\sigma \times C_{\sigma})\Big)\Big/ \simeq$
\end{center}
where $(i_{\sigma,\sigma^{'}}(x),s)\simeq (x,\varphi_{\sigma,\sigma^{'}}(s))$, $[\sigma,\sigma^{'}]\in \EE{}(\YY{})$ 
\end{defn}

\section{Construction of $EG$ and $\partial G$}\label{EG}
Let $G(\YY{})$ be a developable simple complex of group with fundamental group $G$ as defined in \ref{fundaG}. 
For each vertex $v$ of $Y$, the vertex group $G_{v}$ is  relatively hyperbolic to the subgroup ${P}_{v}$.
  Local maps $\varphi_{\sigma,\sigma^{'}}$ are fully quasi-convex embeddings i.e. if $\sigma\subset\sigma'$ 
  then $\varphi_{\sigma,\sigma^{'}}(G_{\sigma'})$
  is fully quasiconvex in $G_{\sigma}$. Then by Remark \ref{quasiconvex-parabolic}, $G_{\sigma}$ is  relatively hyperbolic to the subgroup 
  ${P}_{v}\cap G_{\sigma}$ for each $\sigma\in S(X)$. 
  We call ${P}_{v}\cap G_{\sigma}$ as $P_{\sigma}$.
  By extending the generating set of $G_{\sigma'}$ to a generating set of $G_{\sigma}$, $\varphi_{\sigma,\sigma^{'}}:G_{\sigma'}\to G_{\sigma}$ will 
induce a natural equivariant embeddings between the corresponding Cayley graphs and Augmented spaces. Also, $\varphi_{\sigma,\sigma^{'}}$
naturally extends to the Bowditch boundaries of corresponding local groups.

Let $X$ be the universal covering of $G(\YY{})$ associated to $\iota_{T}$.
Let $\Gamma_{\sigma}$ be the Cayley graph of $G_{\sigma}$ and $\Gamma_{\sigma}^{h}$ be the augmented spaces on which 
$G_{\sigma}$ acts properly discontinuously. Also let $\partial G_{\sigma}$ be the Bowditch Boundary of $G_{\sigma}$ and  $\bbar{\Gamma^{h}_{\sigma}}= \Gamma_{\sigma}^{h}\cup \partial G_{\sigma}$.
\begin{defn}
We define a complex of spaces over $X$, $EG~(resp. ~EG^{h})$ associated to $G(\YY{})$ 
\begin{center}
$EG:=\Bigg(G\times \coprod_{\sigma\in S(Y)}(\sigma\times \Gamma_{\sigma})\Bigg)\bigg/ \simeq$ 

$EG^{h}:=\Bigg(G\times \coprod_{\sigma\in S(Y)}(\sigma\times \Gamma_{\sigma}^{h})\Bigg)\bigg/ \simeq$
\end{center}
where $(g,i_{\sigma,\sigma^{'}}(x),s)\simeq (g\iota_{T}([\sigma,\sigma^{'}])^{-1},x,\varphi_{\sigma,\sigma^{'}}(s))$, 
$[\sigma,\sigma^{'}]\in \EE{}(\YY{})$ and $(gg^{'},x,s)\simeq (g,x,g^{'}s)$, $g^{'}\in G_{\sigma},g\in G$.
\end{defn}
$G$ has natural action on $EG^{h}$ by left multiplication on the first factor.
Also there is a obvious projection map $p:EG^{h}\rightarrow X$ which injectively sends the first two factors and this map is $G$-equivariant. 
\begin{defn}
We define the space 
\begin{center}
$\partial_{stab} G := \Bigg(G\times \coprod_{\sigma\in S(Y)}(\{\sigma\}\times \partial G_{\sigma})\Bigg)\bigg/ \simeq$ 
\end{center}
where $(g,\{\sigma\},s)\simeq (g\iota_{T}([\sigma,\sigma^{'}])^{-1},\{\sigma^{'}\},\varphi_{\sigma,\sigma^{'}}(s))$, $[\sigma,\sigma^{'}]\in \EE{}(\YY{})$ and $(gg^{'},\{\sigma\},s)\simeq (g,\{\sigma\},g^{'}s)$, $g^{'}\in G_{\sigma},g\in G$.

Now we define the boundary of $G$ as 
\begin{center}
$\partial G:=\partial_{stab} G\cup \partial X $
\end{center}
Also we define $\overline{EG^{h}}:= EG^{h}\cup \partial G$.
\end{defn}
Here, we are taking the union of augmented spaces 
(respectively boundaries) corresponding to vertex groups of $X$ and gluing them along the augmented spaces(respectively boundaries) of the local groups accordingly.

$G$ also has natural action on $\partial G$ and $\overline{EG^{h}}$ by left multiplication on the first factor.
In the subsequent section we will try to give a topology on $\overline{EG^{h}}$ such that $\overline{EG^{h}}$ and $\partial G$ will be
compact and action of $G$ will be geometrically finite convergence action. 

 For simplicity of notation we will denote $G_{\sigma}$ as the stabilizer subgroup of the simplex $\sigma$ in $X$(Note that $stab(\sigma)$ is
 actually conjugate of a local group of $G(\YY{})$). It is easy to see that the map $\partial G_{\sigma}\rightarrow \partial G$ is $G_\sigma$- equivariant for every simplex $\sigma$ in $X$.

In the subsequent sections we assume our complex of groups satisfies all the hypothesis of the main theorem. 
Then by \ref{limset}, \ref{conv} and \ref{finite} $\mathcal{G(Y)}$ will 
satisfies \textit{limit set property}, \textit{convergence property} and \textit{finite intersection property}.
\subsection{Domains and Topology}

This subsection is mostly taken from section 4 and 6 from Martin's paper\cite{martin}. Proofs of most of the propositions and theorems will work as it is by adapting to our setting.

\begin{defn}
Let $\xi \in \partial_{stab}G$. We define $\textit{domain}$ of $\xi$, $D(\xi):=span\{\sigma\in S(X) : \xi \in \partial_{stab}G_{\sigma}\}$.
\end{defn}

\begin{prop}{(Propositions 4.2,4.4 of \cite{martin})}\label{domain}
(i) For every vertex $v$, the quotient map $\partial G_v\to\partial G$ is injective,\\
(ii) For every $\xi \in \partial_{stab}G$, $D(\xi)$ is finite convex subcomplex of $X$ uniformly bounded by the acylindricity constant.
\end{prop}

\begin{defn}{($\xi$-family, \cite{martin})}
Let $\xi \in \partial_{stab}G$. A $\xi$-family is defined to be as a collection $\UU{}$ of open sets $U_{v}$ where 
$ v\in V(D(\xi))$ and $U_{v}$ is a neighbourhood of representative of $\xi$ in $\overline{\Gamma^{h}_{v}}$
such that for every two adjacent vertices $v$, $v^{'}$ we have
\begin{center}
$\varphi_{v,e}(\overline{\Gamma^{h}_{e}})\cap U_{v}=\varphi_{v^{'},e}(\overline{\Gamma^{h}_{e}})\cap U_{v^{'}}$,
where $e$ is an edge between $v$ and $v^{'}$
\end{center}
\end{defn}
Next, we  give a topology on $\partial G$ due to Martin \cite{martin}.

Let us choose a basepoint $v_{0} \in X$. For a given point $x\in X(resp.~ \eta \in \partial X)$ we denote $c_{x}(resp.~c_{\eta})$ 
to be the unique geodesic segment(respectively geodesic ray) from $v_{0}$ to $x(resp.~ \eta)$. 
We  denote $D^{\epsilon}(\xi)$ to be the $\epsilon$-neighbourhood of $D(\xi)$ where $\epsilon\in (0,1)$.

A geodesic $c$ is said to be \textit{goes through} (reps. \textit{enters})
$D^{\epsilon}(\xi)$ if $\exists$ $t_{0}, t_{1}$ such that $c(t_{0})\in D^{\epsilon}(\xi),c(t_{1})\in \overline{D^{\epsilon}(\xi)}$ 
and $\forall t>t_{1}$, $c(t)\notin D^{\epsilon}(\xi)$ (respectively if $\exists t_{0}$ such that $c(t_{0})\in D^{\epsilon}(\xi)$).
If $c_{x}$ or $c_{\eta}$ goes through $D^{\epsilon}(\xi)$, the first simplex which is met by $c_{x}$ or $c_{\eta}$
after leaving $D^{\epsilon}(\xi)$ is said to be an \textit{exit simplex} and is denoted by $\sigma_{\xi,\epsilon}(x)$.
For $x\in D^{\epsilon}(\xi)$ we define $\sigma_{\xi,\epsilon}(x):=\sigma_{x}$

\begin{defn}(Martin \cite{martin})
Let $\xi \in \partial_{stab}G, \UU{}$ a $\xi$-family and $\epsilon \in (0,1)$. We define \\
(i) $Cone_{\UU{},\ep{}}(\xi):= \{x\in \overline{X}\setminus D(\xi):$ $c_{x}$ goes 
through $D^{\ep{}}(\xi)$ and for all $v\in V(D(\xi)\cap \sigma_{\xi,\epsilon}(x))$, $ \overline{\Gamma^{h}_{\sigma_{\xi,\epsilon}(x)}}\subset U_{v}$,
in $\overline{\Gamma^{h}_{v}}\}$,\\
(ii) $\widetilde{Cone}_{\UU{},\ep{}}(\xi):= \{x\in \overline{X}:$ $c_{x}$ enters $D^{\ep{}}(\xi)$ and for all $v\in V(D(\xi)
\cap \sigma_{\xi,\epsilon}(x))$, $\overline{\Gamma^{h}_{\sigma_{\xi,\epsilon}(x)}}\subset U_{v}$,in $\overline{\Gamma^{h}_{v}}\}$
\end{defn}
Martin, in \cite{martin}, proved that the cones $Cone_{\UU{},\ep{}}(\xi)$ and $\widetilde{Cone}_{\UU{},\ep{}}(\xi)$
are open sets in $\bbar X$.

\textbf{\underline{Topology on $\overline{EG^{h}}$.}} \\
$\overline{EG^{h}}$ consists of three kind of elements $\tilde{x}\in EG^{h}$, $\eta \in \partial X$ and $\xi \in \partial_{stab}G$.
\begin{itemize}
\item \underline{For $\tilde{x}\in EG^{h}$ }:
We define a basis of neighbourhood of $\tilde{x}$ in $\overline{EG^{h}}$ coming from the 
topology of $EG^{h}$ as a CW complex and denote it by $\OO{}_{\overline{EG^{h}}}(\tilde{x})$.
\item \underline{For $\eta \in \partial X$} :
Let $\OO{}_{\overline{X}}(\eta)$ be the basis of neighbourhood of $\eta$ in $\overline{X}$ and $U\in \OO{}_{\overline{X}}(\eta)$.
we define a neighborhood of $\eta$ in $\overline{EG^{h}}$
\begin{center}
$V_{U}(\eta)=p^{-1}(U\cap X)\cup (U\cap \partial X) \cup \{\xi\in \partial_{stab}G| D(\xi)\subset U\}$
\end{center}
We define, $\OO{}_{\overline{EG^{h}}}(\eta):=\{V_{U}(\eta)|U\in\OO{}_{\overline{X}}(\eta)\}$, 
the basis of neighbourhood of $\eta$ in $\overline{EG^{h}}$.
\item 
\underline{For $\xi \in \partial_{stab}G$}:
Let $\UU{}$ be $\xi$-family and $\ep{}\in(0,1)$. We define four sets around $\xi$ as follows:\\
$W_{1}= \{\tilde{x}\in EG^{h}:p(\tilde{x})=x\in D^{\ep{}}(\xi)~\mbox{and}~ \varphi_{v,\sigma_{x}}(\tilde{x})\in U_{v}~\mbox{for all vertex}~ 
v\in D(\xi)\cap \sigma_{x}\}$,\\
$W_2$ = the set of points in $EG$ whose projection in $X$ belongs to $ Cone_{\UU{},\ep{}}(\xi)$.\\
$W_{3}:= Cone_{\UU{},\ep{}}(\xi)\cap \partial X$,\\
$W_{4}:= \{\xi^{'} \in \partial_{stab} G: D(\xi^{'})\backslash D(\xi)\subset \widetilde{Cone}_{\UU{},\ep{}(\xi)}~\mbox{and}~\xi^{'}\in U_{v},
~\mbox{for all vertex}~ v \in D( \xi )\cap D(\xi^{'})\}$ 
\end{itemize}
We define a neighbourhood around $\xi$ as $W_{\UU{},\ep{}}(\xi):=W_{1}\cup W_{2}\cup W_{3} \cup W_4$.
Let $\OO{}_{\overline{EG^{h}}}(\xi)= \{W_{\UU{},\ep{}}(\xi): \UU{}~ \xi\mbox{-family and}~\ep{}\in(0,1)\}$.
We give $\overline{EG^{h}}$ the topology generated by the sub-basis $\OO{}_{\overline{EG^{h}}}(x)$, $x\in \overline{EG^{h}}$.
In fact Martin showed that $\OO{}_{\overline{EG^{h}}}(x)$ is a basis for this topology. 
Martin, in \cite{martin}, showed that the topology remains equivalent even if we change the base point. From Proposition \ref{domain} the map $\partial G_v\to\partial G$ is injective for all vertex $v$ of $X$, moreover Martin proved that these maps are embedding(Proposition 6.19 \cite{martin}).

For hyperbolic case, that is, if we consider local groups to be hyperbolic and take Gromov boundary instead of Bowditch boundary then 
the Separability, Metrisability and Compactness of $\overline{EG}$ are proved in \cite{martin}. 
The proof requires $X$ to be $CAT(0)$, acylindrical action of $G$ on $X$ and convergence property of the local groups which are true in our case also, 
hence same proofs will work in proving the Separability, Metrisability and Compactness of $\overline{EG^{h}}$.
For instance, to prove sequentially compactness of $\overline{EG^h}$, we take a sequence  $\{x_{n}\}$  of points in $\overline{EG^{h}}$.
Now, due to Theorem 6.17 of \cite{martin}, $EG^h$ is dense in $\overline{EG^{h}}$. So, we can take the sequence $\{x_n\}$ in $EG^h$ and 
let $a_{n}=p(x_{n})$ be its image in $X$. For each $n$, let $\{\sigma^{(n)}_{1},\sigma^{(n)}_{2}...,\sigma^{(n)}_{m(n)}\}$ 
be the path of simplices meet by  the geodesics $[v_{0},a_n]$(note that $\{\sigma^{(n)}_{1}=v_{0}\}$). Then three cases can occur. \\
\underline{Case 1} :
$\{a_{n}\}$'s  contained in finitely many simplices in $X$. Then upto subsequences we can assume for all $n$, $a_{n}$'s 
contained in the interior of a single simplex, $\sigma$ (say). Hence $x_{n}$ will converges to some point of 
$\overline{\Gamma^{h}_{\sigma}}\hookrightarrow \overline{EG^{h}}$.\\
\underline{Case 2} :
Number of simplices in $\{\sigma^{(n)}_{k}\}_n$ is finite for all  $k=1,...,m(n)$. Then upto subsequence 
$<a_{n},a_{n'}>_{v_{0}}\rightarrow \infty$. Hence $\{a_n\}$ converges to $\eta$, where $\eta\in \partial X$. 
From the definition of topology on $\partial G$, it can be proved that $\{x_{n}\}$ converges to $\eta$.\\ 
 \underline{Case 3} : 
 Number of simplices in $\{\sigma^{(n)}_{m}\}$ is infinite for some $m$. Let $m_{0}$ be the first number such that
 the number of simplices in $\{\sigma^{(n)}_{m_{0}}\}$ is infinite. Now upto subsequence we can let 
 $\sigma_{1},\sigma_{2},...,\sigma_{m_{0}-1}$ be the first $m_{0}-1$ number of simplices met by the geodesics $[v_{0},a_n]$. 
 Obviously $\sigma_{m_{0}-1}\subset \sigma^{(n)}_{m_{0}}$ for all $n$. Then by convergence property $\partial G_{\sigma^{(n)}_{m_{0}}}$
 converges to some point $\xi$ in $\partial G_{\sigma_{m_{0}-1}}$. Then from the definition of topology on $\overline{EG^{h}}$,
 it can be shown that  $\{x_{n}\}$ converges to $\xi$.

\begin{thm}(Martin, Theorems 7.12, 7.13 of \cite{martin})
  $\overline{EG^{h}}$ is separable, metrizable and is compact.
\end{thm}

\section{Convergence Group Action of $G$}\label{convergence}

In this section, we describe Martin's strategy (in \cite{martin}) to prove convergence action of $G$ on $\partial G$. It is divided 
into following three propositions. As the proof of these propositions almost remains the same as given by Martin (\cite{martin}), we 
will  not provide the full details but give the ideas and account for where it differs.
\begin{prop}\label{Proposition 1}(by adapting Lemma 9.14 of \cite{martin})
Let $\{g_{n}\}$ be an injective sequence in $G$ and there exists $v_{0}$ and $v_{1}$ such that 
$g_{n}v_{0}=v_{1}$ for infinitely many $n$. Then there exists $\xi^{+},\xi^{-}\in \partial G$ 
and a subsequence $\{g_{n_r}\}$ of $\{g_n\}$ such that for any compact set $K$ in $\partial G\backslash \{\xi^{-}\}$, 
$g_{n_r}K$ convergences to $\xi^{+}$ uniformly.
\end{prop}
We sketch the proof of Proposition \ref{Proposition 1}.
 Without loss of generality, we can take $g_nv_0=v_0$ for infinitely many $n$ and $v_0$ as the base point of the topology
 on $\bbar{EG^h}$. Then $g_n$ stabilizes the vertex space $\Gamma^h_{v_0}$.
 $G_{v_0}$ has convergence action on $\bbar{\Gamma^h_{v_0}}$. Thus, there exists a subsequence of $\{g_n\}$, and points 
 $\xi_{-},\xi_{+}\in\partial G_{v_0}$ such that for every compact set $K_{v_0}$ of $\bbar{\Gamma^h_{v_0}}\setminus\{\xi_{-}\}$, the sequence
 of translates $g_nK_{v_0}$ converge uniformly to $\xi_{+}$. Let $K$ be compact set in $\partial G\backslash \{\xi^{-}\}$ and 
 $\bbar{p}(K)=(\cup_{\xi\in K}D(\xi))\cup(K\cap\partial X)$. We will be applying convergence criterion
 proved by Martin (Corollary 7.16 of \cite{martin}) in order to show that upto a subsequence, $g_nK$ converges uniformly to $\xi_{+}$.
 In order to do that let us first look into simplices in one simplicial neighbourhood of $D(\xi_{+})$.\\
 
 Let $\sigma$ be a simplex in $X$ such that $v_0\in\sigma\cap D(\xi_{+})$. This implies $\partial G_{\sigma}\subset\partial G_{v_0} $.\\ 
 If $\sigma$ is not contained in $D(\xi_{-})$, then $\xi_{-}\notin\partial G_{\sigma}$. Thus, up to a subsequence, 
 convergence action of vertex group $G_{v_0}$  implies that $g_n\partial G_{\sigma}$ converges uniformly to $\xi_{+}$ in $G_{v_0}$.\\ 
 Now, let $\sigma$ be contained in the subcomplex $D(\xi_{-})$ then $\xi_{-}\in \partial G_{\sigma}$.\\
\underline{Case I.} Let $\xi^-$ be parabolic in $\partial G_{v_0}$. Now suppose for some $v\in V(D(\xi^+)\cap D(\xi^-))$ fixed by all $g_n$, 
$g_n.\xi^{-}\rightarrow \xi^{'}(\neq \xi^{+})$. Then we got $\{g_n\}$ and $\xi^{'}(\neq \xi^{+})$ such that 
$g_n.\xi^{-}\rightarrow \xi^{'}$ and $g_n.\tilde{\xi}\rightarrow \xi^{+}$ for all $\tilde{\xi}\neq \xi^-$, 
which contradicts the fact that $\xi^-$ is parabolic. Hence 

1) For all $v\in V(D(\xi^+)\cap D(\xi^-))$ fixed by $g_n$'s and for any compact set $K_v$ in $\partial G_{v}$
up to a subsequence, $g_nK_v$ converges to $\xi^+$ uniformly. 

2) For any simplex $\sigma$ not in $D(\xi^+)\cap D(\xi^-)$ but having a common vertex 
$v\in \sigma\cap(D(\xi^+)\cap D(\xi^-))$ fixed by all $g_n$'s the following holds : if
$\{g_nG_{\sigma}\}$ is  an infinite collection of cosets then, up to a subsequence,
$g_n\partial G_{\sigma}$ converges to $\xi^+$ uniformly by convergence property
for fully quasiconvex subgroups.\\

\underline{Case II.} $\xi^-$ is not parabolic in $\partial G_{v_0}$. Recall, we have taken  $\sigma$ to  be a simplex contained in $D(\xi^-)$ with the
vertex $v_0 \in V(D(\xi^+)\cap D(\xi^-))$, then $\partial G_{\sigma}$ contains at least two points including $\xi^-$, 
otherwise $G_{\sigma}$ would be a parabolic subgroup which implies $\xi^-$ is parabolic in $\partial G_{v_0}$, a contradiction.

 $\bullet$ Let the set of cosets $\{g_nG_{\sigma}:n\geq 1\}$ be infinite.
 For $x\in \partial G_{\sigma}$ other than $\xi^-$, up to a subsequence, $g_nx$ converges to $\xi^+$. 
 This is due to convergence action of $G_v$.  So, $ g_n\partial G_{\sigma}$
 converges uniformly to $\xi_{+}$.\\
 
 $\bullet$ If the set of cosets $\{g_nG_{\sigma}:n\geq 1\}$ is finite then up to a subsequence of $\{g_n\}$, we can take
 $g_n\partial G_{\sigma}=g_N\partial G_{\sigma}$ and $g_n^{-1}g_N$ stabilizes $\sigma$. Replacing $g_n^{-1}g_N$ by $g_n$ we can assume
 $g_n$ stabilizes each $\sigma$ and hence $\xi_{+}\in g_n\partial G_{\sigma}$.

 Suppose $\tau$ is a simplex in $D(\xi_{-})\cap D(\xi_{+})$ fixed pointwise by each element of $\{g_n\}$. For each vertex $v\in\tau$,
  $\xi_{-},\xi_{+}\in \partial G_v$ and due to convergence property of $G_v$ for any compact set $C$ in $\partial G_v\setminus\{\xi_{-}\}$,
 up to a subsequence of $\{g_n\}$, $g_nC$ converges uniformly to $\xi_{+}$. Note that if $\partial G_{\tau}$ is a single point, then $\xi_{-}$
 is a parabolic point. In that case, $\xi_{-}=\xi_{+}$ and for any compact set $C$  in $\partial G_v$,
 up to a subsequence of $\{g_n\}$, $g_nC$ converges uniformly to $\xi_{+}$. 
 Now for any simplex $\sigma$ with a vertex $v\in \sigma\cap\tau$
 and $v\in D(\xi_{+})$, we can continue the above process. Let $A$ be a finite subcomplex in $ D(\xi^+)\cap D(\xi^-)$ such that
 
1) $A$ is fixed by $g_n$'s pointwise.

2) For all simplex $\sigma$ contained in the deleted simplicial neighbourhood of $A$, $g_n\partial G_{\sigma}$ converges to $\xi^+$ uniformly. 

3) For all simplex $\sigma$ in $A$ and for all $v\in V(\sigma\cap A)$,  $g_nK_{\sigma}$ convergences to $\xi^+$ uniformly, 
for any compact sets  $K_{\sigma}$ in $\partial G_{\sigma}\backslash\{\xi^{-}\}$.

Let $K$ be a compact subset in $\partial G\backslash \{\xi^{-}\}$. Now if $K_v=K\cap \partial G_v$ is non-empty for some vertex $v$ of $A$
then  as discussed above, up to a subsequence $g_nK_{v}$ converges to $\xi^+$ uniformly. 
 And for any other point $x$ of $K$, join $v_0$ to $g_nx$ by a geodesic $[v_0,g_nx]$. 
 The exit simplex for the geodesic $[v_0,g_nx]$ from $A$ will lie  $N(A)\setminus A$, where $N(A)$ is one simplicial neighbourhood of $A$. 
 Then by above reason, up to a subsequence, the sequence of translates of exit simplex by $g_n$'s converges to $\xi_{+}$. 
 By convergence criterion proved by Martin (Corollary 7.16 of \cite{martin}), 
 it would imply that $g_{n}K$ convergences to $\xi^{+}$ uniformly. \qed
 
 Suppose $Q$ is a relatively hyperbolic group then it acts on the augmented space $Q^h$ properly discontinuously by isometries. The augmented space
 is proper and hyperbolic. Consider the Bowditch boundary $\partial Q$ of $Q$. 
Let $\xi\in \partial Q$ and $U$ be a neighbourhood of $\xi$ in $\overline{Q^{h}}$. 
Let $K$ be a compact set in $Q^{h}$. Consider a base point $p$ in $Q$. The basis of neighbourhoods of $\xi$ is given by the collection
$V(\xi,r)$ of all $\alpha\in\overline{Q^{h}}$ such that if for some sequences 
$\{x_n\},\{y_n\}$ with $\alpha=[(x_n)],\xi=[(y_n)]$ we have
$\liminf_{i,j\to\infty}(x_i,y_j)_p\geq r$. There exists a sequence $ \{r_{n}\}$
going to infinity such that $V(\xi,r_n)\subsetneqq V(\xi,r_{n+1})$ for all $n$. For all large $n$, $V(\xi,r_n)\subset U$
and the distance between complement of $U$ in $Q$ and closure of $V(\xi,r_n)$ in $Q$ goes to infinity as $n\to\infty$.
Thus, there exists a natural number $N$ such that if some translate of $K$ intersects $V(\xi,r_N)$ then it must be contained in $U$.
Thus, it amounts to say `compact sets fade at infinity' in $Q^h$ i.e. for any $\xi\in \partial Q$,
for any neighbourhood $U$  of $\xi$ in $\overline{Q^{h}}$ and for any compact set $K$ in $Q^h$, there exists a sub neighbourhood $V$
of $\xi$ such that if any $Q$ translate of $K$ intersects $V$ then it must be contained in $U$.\\
Now for a complex of (relatively) hyperbolic groups, in each local groups compact set fade at infinity. Using this,
Martin \cite{martin} proved that compact set in $EG$ fade at infinity. (See Proposition 8.8 of \cite{martin}) 
The same thing hold in our case also where the same proof of Proposition 8.8 goes through.\\
Let $\{g_{n}\}$ be an injective sequence. Using compact set fade at infinity we have
for any compact set $K$ in $EG^h$, up to a subsequence, $g_{n}K$ converges to $ \xi$. This information is used by Martin \cite{martin}
to prove the following two lemmas for complex of hyperbolic groups. The exact proof works in our case also.

\begin{prop}(Lemma 9.15 of \cite{martin})
Let $\{g_{n}\}$ be a injective sequence in $G$. Suppose $\{g_{n}v\}$ is bounded for some(hence any) vertex $v$ and 
there do not exist $v_{0}$ and $v_{1}$ such that $g_{n}v_{0}=v_{1}$ for infinitely many $n$. 
Then there exists $\xi^{+},\xi^{-}\in \partial G$ such that for any compact set $K$ in $\partial G\backslash \{\xi^{-}\}$,
$g_{n}K$ converges to $\xi^{+}$ uniformly.
\end{prop}

\begin{prop} (Lemma 9.16 of \cite{martin})
Let $\{g_{n}\}$ be a injective sequence in $G$ such that $d(g_{n}v_{0},v_{0})\rightarrow \infty$ for some(hence any) vertex $v_{0}$. 
Then there exists $\xi^{+},\xi^{-}\in \partial G$ such that for any compact set $K$ in $\partial G\backslash \{\xi^{-}\}$, $g_{n}K$ 
converges to $\xi^{+}$ uniformly.
\end{prop}

Using above lemmas, we have the following theorem
\begin{thm}(Corollary 9.17 of \cite{martin})\label{conv action}
 $G$ has convergence action on $\partial G$
\end{thm}

\section{Main Theorem}\label{MT}

 Let $\mathcal{G(Y)}$ be a strictly developable simple complex of groups over a finite $M_{\kappa}$-simplicial complex $Y$ with $\kappa\leq 0$
 and satisfying the hypothesis of Theorem \ref{main theorem}. Let $G$ be the fundamental group of $\mathcal{G(Y)}$.

Local groups $G_v$ are relatively hyperbolic implies $G_v$ has convergence action on the Bowditch boundary $\partial G_v$.
Every point on $\partial G_v$ is either conical limit point or bounded parabolic point for the action of $G_v$ on $\partial G_v$.

\begin{lem}(by adapting Lemma 9.18 of \cite{martin})\label{conical point}
 The conical limit points of $G$ are precisely the conical limit points of vertex stabilizers and boundary points of $X$
\end{lem}

\textit{Sketch of Proof:} Consider a conical limit point $\alpha$ in $\partial G_v$ for the action of $G_v$  on $\partial G_v$.
As the map $\partial G_v\to\partial G$ is embedding, the point $\alpha$ is conical limit point for the action of $G_v$ on $\partial G$,
As $G$ has convergence action on $\partial G$, $\alpha$ is conical limit point for the action of $G$ on $\partial G$.

Now let $\eta\in \partial X$. We need to find a sequence $\{{g}_{n}\}$ and $\xi^{+}\neq\xi^{-}\in \partial G$ 
such that $g_{n}\eta \rightarrow \xi^{+}$, $g_{n}\xi^{'} \rightarrow \xi^{-},\forall \xi^{'}\in \partial G \backslash \{\eta\}$.
Since action of $G$ on $X$ is co-compact we can choose a simplex $\sigma$ and a sequence $\{g_n\}$ such that the sequence of simplices $\{g_n \sigma\}$ 
intersect with the geodesic $[v_0,\eta)$. Let $v$ be a vertex of $\sigma$ then $g_{n}\tilde{x}$ converges to $\eta$ for all $\tilde{x}\in \partial G_{v}$. Choose $v$ as the basepoint.  

Consider the sequence $\{g_{n}^{-1}\}$ of group elements. Since $d(v,g_{n}^{-1}v)\rightarrow \infty$,
let $g_{n}^{-1}v$ be converge to $\xi^{-}(\in \partial G)$. Also by Proposition 7.3 except for possibly one elements, 
$g_{n}^{-1}$-translates of boundary points will converges to $\xi^-$.

Suppose $\xi^-\in \partial X$. Note that $<g_{n}^{-1}v,g_{n}^{-1}g_{m}v>_v=d(g_{n} v,v)+d(g_{m}v,g_{n}v)-d(v,g_{m}v)$. 
Now taking projection of $g_{m}v$ and $g_{n}v$ onto geodesic $[v,\eta)$ we can check that $<g_{n}^{-1}v,g_{n}^{-1}g_{m}v>_v$ is 
uniformly bounded for all $m$ and $n$. Hence $g_{n}^{-1}\eta$ cannot converge to $\xi^-$ and we are done. 

Let $\xi^-\in \partial_{stab}G$ and $\tilde{x}\in \partial G_{v}$. Now translating the geodesic $[v,\eta)$ by isometry $g_{n}^{-1}$, we see that the vertex $v$ 
lie uniformly closed to geodesic $[g_{n}^{-1}v,g_{n}^{-1}\eta)$. Hence if $g_{n}^{-1}\eta$ converges to $\xi^-$, 
i.e $g_{n}^{-1}\eta$ and $g_{n}^{-1}\tilde{x}$ converges to same point, then $\xi^-$ must be belongs to $\partial G_{v}$. 
If we can show that there exists $\{h_n\}$ from $G_v$ such that $(h_{n}g_{n}^{-1})v$ does not converge to a point of 
$\partial G_{v}$, then as $(g_{n}h_{n}^{-1})\tilde{x}$ still converges to $\eta$, replacing $\{g_{n}^{-1}\}$ with $\{h_{n}g_{n}^{-1}\}$, we are done.

Let for each $n$, $\sigma_{1}^{(n)}$ be the first simplex met by $[v,g_{n}^{-1}v]$ after leaving $v$. 
Then upto multiplying $g_{n}^{-1}$ by an element from $G_v$ on the left we can let $[v,g_{n}^{-1}v]$ meet a single simplex,
say, $\sigma_{1}$ after leaving $v$. Also let $\tau_{1}$ be the face of $\sigma_1$ which is met by $[v,g_{n}^{-1}v]$ after 
leaving $\sigma_1$. Similarly let $\sigma_{2}^{(n)}$ be the first simplex met by $[v,g_{n}^{-1}v]$ after leaving $\tau_1$. 
Since $G_{\sigma_{1}}$ is fully quasi convex in $G_{\tau_{1}}$, by Lemma \ref{subgroup qc}, $\Gamma^{h}_{\sigma_{1}}$ will be quasi convex in $\Gamma^{h}_{\tau_{1}}$. Choose any $x_{n} \in \Gamma_{\sigma_{2}^{(n)}}$ and let $y_n$ be its projection on $\Gamma^{h}_{\sigma_{1}}$, so $y_n$'s will lie in $\Gamma_{\sigma_{1}}$. Then we can find $\{h_{n}\}\subset G_{\sigma_1}\subset G_{v}$ such that $h_{n}x_{n}$ project to $1$(1 is the identity) for all $n$ since the action of $G_{\sigma_{1}}$ on $\Gamma_{\sigma_{1}}$ is transitive.
Hence $h_{n}\overline{\Gamma^{h}}_{\sigma_{2}^{(n)}}$ do not converge to a point of $\partial G_{\sigma_{1}}$. 
Also since $\Gamma^h_{\sigma_{1}}$ is fixed by all $h_n$'s $h_{n}\overline{\Gamma^{h}}_{\sigma_{2}^{(n)}}$ cannot converge to a point of $\partial G_{\tau_{1}}$. 
Hence by convergence property upto subsequence we can let $\sigma_{2}^{(n)}$ to be constant, say, $\sigma_2$.
We replace $\{g_{n}^{-1}\}$ with $\{h_{n}g_{n}^{-1}\}$. Notice if $G_{\sigma_1}\cap G_{\sigma_2}$ is
finite then the limit of $\{g_{n}^{-1}\tilde{x}\}$, i.e. $\xi^{-}$ cannot be contained in $\partial G_{\sigma_{1}}$ and 
so $\xi^{-}\notin \partial G_v$ because of the convexity of $D(\xi^-)$. If $G_{\sigma_1}\cap G_{\sigma_2}$ is
infinite then we again follow the same process. Since action of $G$ on $X$ is acylindrical after finite number of steps 
intersection of stabilizers will be finite.

The central idea of the following lemma is due to Dahmani \cite{dah}.
\begin{lem}\label{parabolic point}
(i) The image of a bounded parabolic point in vertex stabilizer's boundary is a bounded parabolic for $G$,\\
(ii) The corresponding maximal parabolic subgroup is the image of a maximal parabolic subgroup in the vertex stabilizer.
\end{lem}
\begin{proof}
(i)  Let $\tilde{\xi}$ be a bounded parabolic point of boundary of some vertex stabilizer and $\pi(\tilde{\xi})=\xi$ 
be its image in $\partial G$. We will show $\xi$ is bounded parabolic.

Let  $P=stab(\xi)$ in $G$. Then $P$ stabilizes $D(\xi)$, domain of $\xi$.
Let $\xi_{v_{i}}\in \partial G_{v_{i}}$ be such that $\pi(\xi_{v_{i}})=\xi$, where $\{v_{1},...,v_{n}\}$ is 
the set of vertices of $D(\xi)$. From construction of $\partial_{stab}G$, for each $i=1,,...,n$, $\xi_{v_i}$ is bounded parabolic 
point of $G_{v_i}$ and let $P_{v_i}$ be the maximal parabolic subgroup of $G_{v_i}$ stabilizing $\xi_{v_i}$.
From the construction of $D(\xi)$, $\xi_{v_1},...,\xi_{v_n}$ are the all which are identified to $\xi$. Thus,
$P_{v_i}$ also stabilizes $\{\xi_{v_1},...,\xi_{v_n}\}$ and hence it stabilizes $D(\xi)$. So, $P_{v_i}$ is a subgroup
of $P$. Let  $K_{i}$ be a compact fundamental domain in $\partial G_{v_{i}}\setminus\{\xi_{v_{i}}\}$
for co-compact action of $P_{v_{i}}$ on $\partial G_{v_{i}}\setminus\{\xi_{v_{i}}\}$.
Let $N(D(\xi))$ be one open simplicial neighbourhood of $D(\xi)$ in $X$
and $S(N(D(\xi))\setminus D(\xi))$ be collection of simplices in $N(D(\xi))$  of $D(\xi)$ that is not 
contained in $D(\xi)$. Let $S_{i}:=\{\sigma\in S(N(D(\xi))\setminus D(\xi)): \partial G_{\sigma}\cap K_{i}\neq \emptyset\}$.

We claim that $\bigcup\limits_{i=1}^{n} PS_{i}=S(N(D(\xi))\backslash D(\xi))$. 
Let $\sigma\in S_i$ and $p\in P$. As $P$ stabilizes $D(\xi)$,
then $p\sigma\in  S(N(D(\xi))\setminus D(\xi))$. Conversely, let $\sigma\in (N(D(\xi))\setminus D(\xi))$
and $v_i\in D(\xi)\cap\sigma$. Then $\partial G_{\sigma}\subset \partial G_{v_i}\setminus\{\xi_{v_i}\}$.
But, $K_{i}$ is a fundamental domain for $P_{v_{i}}$ hence there exists $p\in P_{v_{i}}\hookrightarrow P$ such that
$p\partial G_{\sigma}=\partial G_{p\sigma}$ intersect with $K_{i}$. So, $p\sigma\in S_i$ and this proves our claim. 

$D(\xi)$ is a finite closed convex subspace of the $CAT(0)$ space $X$ and is stabilized by $P$. Hence, $P$ has a fix point, say $\{x_0\}$, in $D(\xi)$.
The topology on $\partial G$ is independent of base point.
Let us take $x_{0}$ to be the base point for the topology of $\partial G$.  For $x\in \bbar{X}\setminus D(\xi)$, there exists $0<\epsilon_{x}<1$
such that $x\in \bbar{X}\setminus D_{\epsilon_x}(\xi)$. Let $\sigma_{x,\epsilon_x}\in S(N(D(\xi))\setminus D(\xi))$
denote the exit simplex for the geodesic $[x_0,x]$.\\
$\bullet$ For each $i$, let $T_{i}:=\{x\in \bbar{X}\setminus D(\xi) : \sigma_{x,\epsilon_x}\in S_{i}\}$.\\
$\bullet$ Let $K^{'}_{i}:=\{ \alpha \in \partial G : D(\alpha)\cap {T_{i}}\neq\emptyset\}$ 
and $\overline{K^{'}_{i}}$ be its closure in $\partial G$.\\
For each $i$, $K_{i}\cup \overline{K^{'}_{i}}$ being closed is  compact in $\partial G$.
 We claim $\xi\notin (K_{i}\cup \overline{K^{'}_{i}})$ for all $i$
and $\bigcup\limits_{i=1}^{n}(K_{i}\cup \overline{K^{'}_{i}})$ is a compact fundamental domain for action of $P$ on $\partial G\backslash \{\xi\}$.

\underline{Claim 1. $\xi\notin (K_{i}\cup \overline{K^{'}_{i}})$.}

For each $i$, $K_i\subset \partial G_{v_i}\setminus\{\xi_{v_i}\}$ implies  $\xi\notin K_{i}$ and $D(\xi)\cap T_i=\phi$
implies $\xi\notin K^{'}_{i}$. 
Now if possible let $\{\alpha_{m}\}$ be a sequence in $K^{'}_{i}$ for some $i$
such that $\alpha_{m}\rightarrow \xi$. By the definition of the topology on $\partial G$, 
$D(\alpha_{m})\setminus D(\xi)\subset \widetilde{Cone}_{\UU{},\ep{}}(\xi)$ for any $\xi$-family $\UU{}$ and $0<\ep{}<1$.
Let  $x_{m}\in D(\alpha_{m})\cap T_i$ then by definition of $K_i^{'}$, 
$\partial G_{\sigma_{x_m},\epsilon_{x_m}}\cap K_i\neq \phi$ for all $m$.
Also, $\partial G_{\sigma_{x_m},\epsilon_{x_m}}\subset \partial G_{v_i}\setminus\{\xi_{v_i}\}$, by convergence Property 
$\partial G_{\sigma_{x_{m},\epsilon_{x_m}}}\rightarrow \xi_{v_{i}}$ uniformly. This implies $\xi_{v_{i}}\in K_{i}$, which is a contradiction.

\underline{Claim 2. $\bigcup\limits_{i=1}^{n}P(K_{i}\cup \overline{K^{'}_{i}})=\partial G\backslash \{\xi\}$.}

Let $\alpha(\neq \xi)\in \partial G$. If $\alpha\in \partial X$ then the claim is true since $x_{0}$ is fixed by $P$
and $\bigcup\limits_{i=1}^{n} PS_{i}=S(N(D(\xi))\backslash D(\xi))$. For $\alpha \in \partial_{stab} G$
we will divide the proof of the claim into two cases.

Case 1. $D(\alpha)\cap D(\xi)\neq \emptyset$. Then $\alpha\in \partial G_{v_{i}}$ for some $v_i\in D(\alpha)\cap D(\xi)$.
 $\alpha\neq \xi_{v_{i}}$, now since $K_{i}$ is a fundamental domain for the action of $P_{v_{i}}$ in $\partial G_{v_{i}}\setminus\{\xi_{v_i}\}$,
 there exists $x\in K_{i}$ and  $p\in P_{v_{i}}\hookrightarrow P$ such that $\alpha=px\in PK_i$

Case 2. $D(\alpha)\cap D(\xi)= \emptyset$. Let $x\in D(\alpha)$ and $\sigma_{x,\epsilon_x}\in S(N(D(\xi))\setminus D(\xi))$ be the exit simplex
for the geodesic $[x_0,x]$ in $X$. As $\bigcup\limits_{i=1}^{n} PS_{i}=S(N(D(\xi))\setminus D(\xi))$ and $P$ fixes $x_0$
there exists $p\in P$ such that $\sigma_{px,\epsilon_x}=p\sigma_{x,\epsilon_x}\in S_i$ for some $i$. So, $px\in T_i$ and $px\in pD(\alpha) =D(p\alpha)$.
So, $p\alpha\in K'_i$ and hence $\alpha\in PK'_i$.\\
 (ii) Let $\tilde{\xi}$ be a bounded parabolic point of boundary of some vertex stabilizer and $\pi(\tilde{\xi})=\xi$ 
be its image in $\partial G$, with $P=stab(\xi)$ in $G$. Then $P$ stabilizes $D(\xi)$ and it fixes a point $x_0\in D(\xi)$.
Let $\sigma$ be the simplex in $D(\xi)$ containing $x_0$ in the interior. From the definition of action of $G$ on $X$, if some element of $G$
fixes an interior point of a simplex then it fixes the whole simplex pointwise. So, $P$ fixes $\sigma$ pointwise. Without loss of generality,
we can take $x_0$ to be a vertex $v_i$ of $\sigma$. Thus $P$ fixes ${\xi}_{v_i}$ and hence $P=P_{v_i}$.

\end{proof}

\textbf{Proof of Theorem \ref{main theorem}:} From Lemma \ref{conv action}, $G$ has a convergence action on compact metrizable space $\partial G$.
The limit points are either conical (by Lemma \ref{conical point}) or bounded parabolic (by Lemma \ref{parabolic point}). Hence, by Theorem
\ref{yaman thm}(due to Yaman, \cite{yaman}), $G$ is hyperbolic relative to $\mathcal{P}$, where $\mathcal{P}$ is the collection of the images of ${P}_{v}$ in $G$
under the natural embedding $G_v\hookrightarrow G$. This embedding extends to a $G_v$-equivariant embedding 
$\partial G_v\hookrightarrow \partial G$. Hence, the limit set for the action of $G_v$ on $\partial G$  is $\partial G_v$ and this action
is geometrically finite implies that $G_v$ is  quasiconvex in $G$. For $\xi\in\partial G_v\subset\partial G$, the domain $D(\xi)$ of $\xi$
is finite implies that $\cap_{n\geq 1}g_n\partial G_v$ is empty for any sequence of infinite distinct left cosets $g_nG_v$
of $G_v$ in $G$. Thus, $G_v$ is fully quasiconvex in $G$. Local groups are fully quasiconvex in vertex groups implies that local
groups are fully quasiconvex in $G$.

\end{document}